\documentclass[12pt]{amsart}

\usepackage{etex}

\usepackage[utf8]{inputenc}
\usepackage{pstricks-add}

\usepackage{amscd,amsmath,amsthm,amssymb,graphics}
\usepackage{pstcol,pst-plot,pst-3d}
\usepackage{lmodern,pst-node}
\usepackage{multicol}
\usepackage{epic,eepic}
\usepackage{amsfonts,amssymb,amscd,amsmath,enumitem,verbatim}
\psset{unit=0.7cm,linewidth=0.8pt,arrowsize=2.5pt 4}

\newpsstyle{fatline}{linewidth=1.5pt}
\newpsstyle{fyp}{fillstyle=solid,fillcolor=verylight}

\def\NZQ{\Bbb}               
\def\NN{{\NZQ N}}
\def\QQ{{\NZQ Q}}
\def\ZZ{{\NZQ Z}}

\def\frk{\frak}               

\def\Phi{{\frk n}}
\def\Phi{{\frk N}}
\def\MP{{\mathcal P}}

\def\MG{{\mathcal G}}
\def\MA{{\mathcal A}}

\def\MB{{\mathcal B}}

\usepackage[bookmarks=true]{hyperref}
\def\opn#1#2{\def#1{\operatorname{#2}}} 
\opn\chara{char} \opn\length{\ell} \opn\pd{pd} \opn\rk{rk}
\opn\projdim{proj\,dim} \opn\injdim{inj\,dim} \opn\rank{rank}
\opn\depth{depth} \opn\grade{grade} \opn\height{height}
\opn\embdim{emb\,dim} \opn\codim{codim}

\opn\Tr{Tr} \opn\bigrank{big\,rank}
\opn\superheight{superheight}\opn\lcm{lcm}
\opn\trdeg{tr\,deg}
\opn\reg{reg} \opn\lreg{lreg} \opn\ini{in} \opn\lpd{lpd}
\opn\size{size}\opn\bigsize{bigsize}
\opn\cosize{cosize}\opn\bigcosize{bigcosize}
\opn\sdepth{sdepth}\opn\sreg{sreg}
\opn\link{link}\opn\fdepth{fdepth}\opn\rev{rev}
\opn\div{div} \opn\Div{Div} \opn\cl{cl} \opn\Cl{Cl}
\let\epsilon\varepsilon
\let\phi=\varphi
\let\kappa=\varkappa
\opn\Spec{Spec} \opn\Supp{Supp} \opn\supp{supp} \opn\Sing{Sing}
\opn\Ass{Ass} \opn\Min{Min}\opn\Mon{Mon} \opn\dstab{dstab} \opn\astab{astab}
\opn\Syz{Syz}
\opn\Ann{Ann} \opn\Rad{Rad} \opn\Soc{Soc}
\opn\Im{Im} \opn\Ker{Ker} \opn\Coker{Coker} \opn\Am{Am}
\opn\Hom{Hom} \opn\Tor{Tor} \opn\Ext{Ext} \opn\End{End}
\opn\Aut{Aut} \opn\id{id}

\opn\nat{nat}
\opn\pff{pf}
\opn\Pf{Pf} \opn\GL{GL} \opn\SL{SL} \opn\mod{mod} \opn\ord{ord}
\opn\Gin{Gin} \opn\Hilb{Hilb}\opn\sort{sort}
\opn\initial{init}
\opn\ende{end}
\opn\height{height}
\opn\type{type}
\opn\aff{aff} \opn\con{conv} \opn\relint{relint} \opn\st{st}
\opn\lk{lk} \opn\cn{cn} \opn\core{core} \opn\vol{vol}
\opn\link{link} \opn\star{star}\opn\lex{lex}
\opn\gr{gr}

\def\pot#1#2{#1[\kern-0.28ex[#2]\kern-0.28ex]}

\opn\dirlim{\underrightarrow{\lim}}
\opn\inivlim{\underleftarrow{\lim}}
\let\union=\cup
\let\sect=\cap

\let\tensor=\otimes
\let\iso=\cong

\let\to=\rightarrow

\def\Implies{\ifmmode\Longrightarrow \else
        \unskip${}\Longrightarrow{}$\ignorespaces\fi}
\def\implies{\ifmmode\Rightarrow \else
        \unskip${}\Rightarrow{}$\ignorespaces\fi}
\def\iff{\ifmmode\Longleftrightarrow \else
        \unskip${}\Longleftrightarrow{}$\ignorespaces\fi}

\let\:=\colon
 \theoremstyle{plain}
\newtheorem{Theorem}{Theorem}[section]
 \newtheorem{Lemma}[Theorem]{Lemma}
 \newtheorem{Corollary}[Theorem]{Corollary}

 \newtheorem{Conjecture}[Theorem]{Conjecture}
 
 \theoremstyle{definition}

 \newtheorem{Example}[Theorem]{Example}

\let\epsilon\varepsilon
\let\kappa=\varkappa
\opn\dis{dis}
\def\pnt{{\raise0.5mm\hbox{\large\bf.}}}

\opn\Lex{Lex}

\begin{document}
\title{Isotonian Algebras}
\author {Mina Bigdeli, J\"urgen Herzog, Takayuki Hibi, Ayesha Asloob Qureshi and  Akihiro Shikama}

\address{Mina Bigdeli, Department  of Mathematics,  Institute for Advanced Studies in Basic Sciences (IASBS),
45195-1159 Zanjan, Iran} \email{m.bigdelie@iasbs.ac.ir}

\address{J\"urgen Herzog, Fachbereich Mathematik, Universit\"at Duisburg-Essen, Fakult\"at f\"ur Mathematik, 45117
Essen, Germany} \email{juergen.herzog@uni-essen.de}

\address{Takayuki Hibi, Department of Pure and Applied Mathematics, Graduate School of Information Science and Technology,
Osaka University, Suita, Osaka 565-0871, Japan}
\email{hibi@math.sci.osaka-u.ac.jp}

\address{Ayesha Asloob Qureshi, Department of Pure and Applied Mathematics, Graduate School of Information Science and Technology,
Osaka University, Suita, Osaka 565-0871, Japan}
\email{ayesqi@gmail.com}
\address{Akihiro Shikama, Department of Pure and Applied Mathematics, Graduate School of Information Science and Technology,
Osaka University, Suita, Osaka 565-0871, Japan}
\email{a-shikama@cr.math.sci.osaka-u.ac.jp}

\begin{abstract}
To a pair $P$ and $Q$ of finite posets we attach the toric ring $K[P,Q]$ whose generators are in bijection to the isotone maps from $P$ to $Q$. This class of algebras, called isotonian, are natural generalizations of the so-called Hibi rings. We determine the Krull dimension of these algebras and for particular classes of posets $P$ and $Q$ we show that $K[P,Q]$ is normal and that their defining ideal admits a quadratic Gr\"obner basis.
\end{abstract}
\thanks{The first author thanks the University of Duisburg-Essen for the hospitality.  She also   wants to thank the University of Osnabr\"uck and IASBS for partial financial support during the preparation of this paper. The fourth author was supported by JSPS Postdoctoral Fellowship for Overseas Researchers FY2014.}

\maketitle
\section*{Introduction}
Let $K$ be a field, and let $L$ be a finite distributive lattice.  In 1987, the third author \cite{Hi}
introduced the $K$-algebra $K[L]$ which nowadays is called the Hibi ring of the distributive lattice $L$.
The $K$-algebra $K[L]$ is generated over $K$ by the elements $\alpha\in L$ with defining relations
$\alpha\beta=(\alpha\wedge \beta)(\alpha\vee\beta)$ with $\alpha,\beta\in L$.
In that early paper, Hibi also showed that $K[L]$ is a normal Cohen--Macaulay domain.

One remarkable fact is that  $K[L]$ may be viewed as  a toric ring. This can be seen by using Birkhoff's fundamental theorem from 1937 which says that each finite distributive lattice is the ideal lattice of a finite poset $P$. Indeed, the subposet of $L$, induced by the join-irreducible elements of $L$, is the poset $P$ whose ideal lattice $\mathcal{I}(P)$ is the given distributive lattice $L$. Having the poset $P$ of join-irreducible elements of $L$ at our disposal, we can write $K[L]$ as the $K$-algebra generated over $K$  by the monomials $u_I=\prod_{p\in I}x_p\prod_{p\not\in I}y_p\in K[\{x_p,y_p\}_{p\in P}]$ with $I\in \mathcal{I}(P)$. Hibi, in his classical paper also showed that the Krull dimension of $K[L]$ is equal to the cardinality $|P|$ of $P$.

 Birkhoff's  theorem  can also be phrased as follow: let $P$ an $Q$ be finite posets. We denote by $\Hom(P,Q)$ the set of order preserving maps, also called isotone maps. Observe that $\Hom(P,Q)$ is again a finite poset, by setting $\varphi\leq \psi$ if and only if $\varphi(p)\leq \psi(p)$ for all $p\in P$.  let  $L$ be  a distributive lattice with $P$ its subposet of join irreducible elements. Birkhoff's theorem is equivalent to saying that there  is a natural isomorphism of posets $L\iso \Hom(P,[2])$. Here for an integer $n$ we denote by $[n]$  the totally ordered set $\{1<2<\cdots<n\}$.

 The set $\MP$ of finite posets together with the isotone maps forms a category,  first considered in \cite{FGH}. In the same paper the authors introduced the ideal $L(P,Q)$ which  is generated by the monomials
\[
u_{\varphi}=\prod_{p\in P}x_{p,\varphi(p)}\quad \text{with}\quad p\in P.
\]
In the special cases $P=[n]$ or $Q=[n]$, the ideals $L(P,Q)$  first  appeared in the work \cite{EHM} of Ene, Mohammadi and the second author. In the sequel the algebraic and homological properties of the ideals $L(P,Q)$ have been subject of further investigations in the papers \cite{HSQ} and \cite{KKM}.

Here we are interested in the algebras $K[P,Q]$ which are the toric rings generated over $K$ by the monomials $u_\varphi$ with $\varphi\in \Hom(P,Q)$. We call these algebras {\em isotonian} because their  generators correspond bijectively to the isotone maps from $P$ to $Q$. In the special case that $Q=[2]$,  we obtain the classical Hibi rings. Accordingly, one would expect that isotonian algebras share all the nice properties of Hibi rings. Their Krull dimension can be determined. In Theorem~\ref{dimension}  it is shown that $\dim K[P,Q]=|P|(|Q|-s)+rs-r+1$, where $r$ is the number of connected components of $P$ and $s$ is the number of connected components of $Q$. The proof of this theorem also shows that the algebraic  variety whose coordinate ring is $K[P,Q]$ is birationally equivalent to the Segre product of suitable copies of affine  spaces.

Hibi rings are normal and Cohen--Macaulay. Do the same properties hold true for isotonian algebras? In Corollary~\ref{normal} it is shown that this is the case when the Hasse diagram of $P$ is a forest.  But this is also the case when $Q=[n]$, as shown in \cite[Corollary 4.3]{EHM}. Corollary~\ref{normal} is a straightforward  consequence of a more general fact. Indeed, in Theorem~\ref{normaltheorem} the following result is proved: let $P'$ be the poset which is obtained from $P$ by adding an element $p'$ to $P$ which has a unique upper or lower neighbor in $P$. Then $K[P,Q]$ is normal if and only if $K[P',Q]$ is normal. Based on these results and on computational evidence  we are lead to conjecture that  all isotonian algebras are normal Cohen--Macaulay domains.  One way to prove this conjecture in general would be to show that there exists a term order such that the initial ideal of the defining ideal  $J_{P,Q}$ of $K[P,Q]$  is squarefree. By a theorem of Sturmfels \cite[Chapter 8]{St} this would imply that $K[P,Q]$ is normal, and by a theorem of Hochster \cite[Theorem 1]{Ho} this in turn implies Cohen--Macaulayness. In the last  section of this paper we show in Theorem~\ref{Boston} that if $P$ is a chain and $Q$ is a rooted or co-rooted poset, then $J_{P,Q}$ has a quadratic (and hence also a squarefree)  Gr\"obner basis with respect to the reverse lexicographic order induced by a canonical labeling of the variables. In general, $J_{P,Q}$ may contain generators of arbitrarily high degree. This happens  if $Q$ contains as an induced subposet, what we call a poset cycle.  On the other hand, we conjecture that $J_{P,Q}$ is quadratically generated if and only if $Q$ does not contain any induced poset cycle of length greater than 4.

\section{Operations on posets and  the $K$-algebra $K[P,Q]$}

 Let $P$ and $Q$ be finite posets. A map $\varphi\:  P \to Q$ is called {\em isotone} (order preserving), if $\varphi(p)\leq \varphi(p')$ for all $p,p'\in P$ with $p<p'$. The set of all isotone maps from $P$ to $Q$ is denoted by $\Hom(P,Q)$. Obviously, if $P,Q$ and $R$ are finite posets  and  $\varphi\in \Hom(P,Q)$ and $\psi\in \Hom(Q,R)$, then  $\psi\circ \varphi\in \Hom(P,R)$.  We denote by $\mathcal{P}$ the category whose objects are finite posets and whose morphisms are isotone maps. We note that $\Hom(P,Q)$ is again a poset with $\varphi\leq \psi$ for $\varphi, \psi\in \Hom(P,Q)$ if and only if $\varphi(p)\leq \psi(p)$ for all $p\in P$. Thus, $\Hom(P,\_ )\: \MP\to \MP$  is a covariant and $\Hom(\_\, ,Q)\: \MP\to \MP$ a contravariant functor.

 Let $P\in \MP$, and let $p_1,p_2\in P$. One says that $p_2$ {\em covers} $p_1$ if $p_1<p_2$, and there is no $p\in P$ with $p_1<p<p_2$.  We define the  graph $G(P)$ on the vertex set $P$ as follows:  a $2$-element  subset   $\{p_1,p_2\}$ is an edge of $G(P)$ if and only if $p_2$ covers $p_1$ or $p_1$ covers $p_2$. The graph $G(P)$ is the underlying graph of the so-called {\em Hasse diagram} of $P$ which may be viewed as a directed graph whose  edges are  the ordered pairs $(p_1,p_2)$, where $p_2$ covers $p_1$.

We say that $P$ is {\em connected},  if $G(P)$ is a connected graph. Given two posets $P_1$ and $P_2$, the sum $P_1+P_2$ is defined to be the disjoint union of the elements  of $P_1$ and $P_2$ with $p\leq q$ if and only if $p,q\in P_1$ or $p,q\in P_2$ and $p\leq q$ in the corresponding posets $P_1$ or $P_2$. Then it is clear that any $P\in \MP$ can be written as  $P=\sum_{i=1}^rP_i$ where each $P_i$ is a connected subposet of $P$. The subposets $P_i$ of $P$ are called the {\em connected components} of $P$.

\begin{Example}
Let $P$ be the poset displayed in Figure~\ref{Hi}.
\begin{figure}[!htbp]
\begin{center}
\psset{xunit=.8cm,yunit=.8cm,algebraic=true,dimen=middle,dotstyle=o,dotsize=5pt 0,linewidth=0.8pt,arrowsize=3pt 2,arrowinset=0.25}
\begin{pspicture*}(6.8,-.9)(12.3,2.6)
\psline(8.38,1.84)(9.4,-0.24)
\psline(10.48,1.84)(9.4,-0.24)
\rput[tl](9.26,-0.64){\begin{scriptsize}1\end{scriptsize}}
\rput[tl](8.26,2.4){\begin{scriptsize}2\end{scriptsize}}
\rput[tl](10.42,2.38){\begin{scriptsize}3\end{scriptsize}}
\begin{scriptsize}
\psdots[dotstyle=*,linecolor=black](8.38,1.84)
\psdots[dotstyle=*,linecolor=black](9.4,-0.24)
\psdots[dotstyle=*,linecolor=black](10.48,1.84)
\end{scriptsize}
\end{pspicture*}
\caption{Poset $P$}
\label{Hi}
\end{center}
\end{figure}

We identify an isotone map $\varphi\: P \to P$ with $\varphi(p_i)=p_{j_i}$ ($i=1,2,3$) with the sequence $j_1j_2j_3$. With the notation introduced, the elements of $\Hom(P,P)$ are:
\[
111, 112, 121, 113, 131, 122, 123, 132, 133, 222, 333.
\]
The poset $\Hom(P,P)$ is displayed in Figure~\ref{hom}.
\begin{figure}[!htbp]
\begin{center}
\psset{xunit=.8cm,yunit=.8cm,algebraic=true,dimen=middle,dotstyle=o,dotsize=5pt 0,linewidth=0.8pt,arrowsize=3pt 2,arrowinset=0.25}
\begin{pspicture*}(1.5,-5.)(12.5,4.)
\psline(7.032209159550816,-3.2578631247540226)(5.888546620539439,-2.1786322780813188)
\psline(7.032209159550816,-3.2578631247540226)(8.175871698562192,-2.275279816589322)
\psline(7.096640851889485,-1.0832935083239474)(5.888546620539439,-2.1786322780813188)
\psline(7.096640851889485,-1.0832935083239474)(8.175871698562192,-2.275279816589322)
\psline(9.36785800682757,-1.0832935083239474)(8.175871698562192,-2.275279816589322)
\psline(4.889855389290069,-1.1799410468319507)(5.888546620539439,-2.1786322780813188)
\psline(5.260337620237416,0.6402542617354456)(4.889855389290069,-1.1799410468319507)
\psline(5.260337620237416,0.6402542617354456)(7.096640851889485,-1.0832935083239474)
\psline(8.932944083541553,0.4147433385501045)(7.096640851889485,-1.0832935083239474)
\psline(8.932944083541553,0.4147433385501045)(9.36785800682757,-1.0832935083239474)
\psline(6.999993313381482,0.6080384155661112)(7.096640851889485,-1.0832935083239474)
\psline(5.260337620237416,0.6402542617354456)(6.565079390095465,1.9449960315934907)
\psline(6.565079390095465,1.9449960315934907)(8.932944083541553,0.4147433385501045)
\psline(6.999993313381482,0.6080384155661112)(8.240303390900861,2.895363493588857)
\psline(8.240303390900861,2.895363493588857)(6.565079390095465,1.9449960315934907)
\rput[tl](6.758374467111474,-3.560561201870703){\begin{scriptsize}111\end{scriptsize}}
\rput[tl](5.24422969715275,-2.320251124351327){\begin{scriptsize}121\end{scriptsize}}
\rput[tl](8.353058852493534,-2.21927355097325){\begin{scriptsize}112\end{scriptsize}}
\rput[tl](4.103919619633369,-1.0999410468319507){\begin{scriptsize}131\end{scriptsize}}
\rput[tl](9.541692699266915,-0.9005380467312769){\begin{scriptsize}113\end{scriptsize}}
\rput[tl](6.710050697857472,-1.5148549701179657){\begin{scriptsize}122\end{scriptsize}}
\rput[tl](9.056778775980899,0.629591847194389){\begin{scriptsize}123\end{scriptsize}}
\rput[tl](4.46078081326715,0.763621848201128){\begin{scriptsize}132\end{scriptsize}}
\rput[tl](7.161072544228156,0.8225693967767){\begin{scriptsize}222\end{scriptsize}}
\rput[tl](5.6024960819307625,2.0238594162708286){\begin{scriptsize}133\end{scriptsize}}
\rput[tl](8.401922237170872,3.1047664936895307){\begin{scriptsize}333\end{scriptsize}}
\begin{scriptsize}
\psdots[dotstyle=*,linecolor=black](7.032209159550816,-3.2578631247540226)
\psdots[dotstyle=*,linecolor=black](5.888546620539439,-2.1786322780813188)
\psdots[dotstyle=*,linecolor=black](8.175871698562192,-2.275279816589322)
\psdots[dotstyle=*,linecolor=black](7.096640851889485,-1.0832935083239474)
\psdots[dotstyle=*,linecolor=black](9.36785800682757,-1.0832935083239474)
\psdots[dotstyle=*,linecolor=black](4.889855389290069,-1.1799410468319507)
\psdots[dotstyle=*,linecolor=black](5.260337620237416,0.6402542617354456)
\psdots[dotstyle=*,linecolor=black](8.932944083541553,0.4147433385501045)
\psdots[dotstyle=*,linecolor=black](6.999993313381482,0.6080384155661112)
\psdots[dotstyle=*,linecolor=black](6.565079390095465,1.9449960315934907)
\psdots[dotstyle=*,linecolor=black](8.240303390900861,2.895363493588857)
\end{scriptsize}
\end{pspicture*}
\caption{$\Hom(P,P)$}\label{hom}
\end{center}
\end{figure}

\end{Example}

The {\em  product} $P_1\times P_2$  of $P_1$ and $P_2$ is the poset whose elements are the pairs $(p_1,p_2)$ with $p_1\in P_1$ and $p_2\in P_2$. The order  relations in $P_1\times P_2$ are defined componentwise. For $P_1\times P_2\times \cdots \times P_s$ we also write $\prod_{i=1}^sP_i$.

\medskip
In the next lemma we present two obvious (but useful) rules of the $\Hom$-posets.

\begin{Lemma}
\label{sum}
Let $P,P_1,P_2,\ldots,P_r$ and $Q,Q_1,Q_2,\ldots,Q_s$ be finite posets, and assume that $P$ is connected. Then
\begin{itemize}
\item[{\em (a)}] $\Hom(\sum_{i=1}^rP_i,Q)\iso \prod_{i=1}^r\Hom(P_i,Q)$;
\item[{\em (b)}] $\Hom(P, \sum_{i=1}^sQ_i)\iso \sum_{i=1}^s\Hom(P,Q_i)$.
\end{itemize}
\end{Lemma}

\medskip
We now introduce the {\em isotonian algebra} $K[P,Q]$ attached to a pair $P$, $Q$ of finite posets. For this purpose we fix a field $K$, and consider the polynomial ring  over $K$ in the variables $x_{p,q}$  with $p\in P$ and $q\in Q$. Then $K[P,Q]$ is the toric ring generated over $K$ by the monomials
\[
u_\varphi =\prod_{p\in P}x_{p,\varphi(p)}.
\]

Let $R_1=K[f_1,\ldots,f_r]\subset K[x_1,\ldots,x_n]$ and $R_2=K[g_1,\ldots,g_s]\subset K[y_1,\ldots,y_m]$ be two standard graded $K$-algebras. Then
\[
R_1\tensor R_2= K[f_1,\ldots,f_r, g_1,\ldots,g_s]
\]
is the tensor product of $R_1$ and $R_2$ over $K$, while
\[
R_1*R_2=K[\{f_ig_j\:\; i=1,\ldots,r,j=1,\ldots,s\}]
\]
is the Segre product of $R_1$ and $R_2$.

The following isomorphisms are immediate consequences of Lemma~\ref{sum} and the definition of isotonian algebras.

\begin{Lemma}\label{segre}
Let $P,P_1,P_2,\ldots,P_r$ and $Q,Q_1,Q_2,\ldots,Q_s$ be finite posets and assume that $P$ is connected. Then
\begin{itemize}
\item[{\em (a)}] $K[\sum_{i=1}^rP_i,Q]\iso K[P_1,Q]*K[P_2,Q]*\cdots *K[P_r,Q]$;
\item[{\em (b)}] $K[P, \sum_{i=1}^sQ_i]\iso K[P,Q_1]\tensor K[P,Q_2]\tensor \cdots \tensor K[P,Q_s]$.
\end{itemize}
\end{Lemma}

As a first consequence we obtain

\begin{Corollary}
\label{normalconnected}
Let $P$ be a finite poset with connected components $P_1,P_2,\ldots,P_r$ and  $Q$ a finite poset with connected components $Q_1,Q_2,\ldots,Q_s$. Then $K[P,Q]$ is normal if all $K[P_i,Q_j]$ are normal. In particular, if this is the case, then $K[P,Q]$ is also Cohen--Macaulay.
\end{Corollary}

\begin{proof}
It is a well-known fact that the Segre product or the tensor product of normal standard graded toric rings is normal.
Thus the assertion follows. The Cohen--Macaulayness  of $K[P,Q]$ is then the consequence of Hochster's theorem \cite[Theorem 1]{Ho}.
\end{proof}

\medskip
Let $P\in \MP$. The {\em dual poset} $P^\vee$ of $P$ is the poset whose underlying set coincides with that of $P$ and whose order relations are reversed. In other words, $p<p'$ in $P$ is equivalent to $p'<p$ in $P^\vee$.

Since $\varphi \: P\to Q$ is isotone if and only if $\varphi \: P^\vee\to Q^\vee$ is isotone, we have

\begin{Lemma}
\label{obvious}
Let $P$ and $Q$ be finite posets. Then the $K$-algebras $K[P,Q]$ and $ K[P^\vee,Q^\vee]$ are isomorphic.
\end{Lemma}

\section{The dimension of $K[P,Q]$}

In this section we compute  the dimension of the algebra $K[P,Q]$. The result is given in

\begin{Theorem}
\label{dimension}
Let $P$ and $Q$ be finite posets and let $r$ be the number of connected components of $P$ and $s$ be the number of connected components of $Q$. Then $\dim K[P,Q]=|P|(|Q|-s)+rs-r+1$.
\end{Theorem}
 \begin{proof}
 By using Lemma~\ref{segre} and the fact that for any two standard graded $K$-algebras $R$ and $S$,  $\dim R \tensor S = \dim R + \dim S$ and $\dim R * S = \dim R + \dim S-1$ (see \cite[Theorem 4.2.3]{GW}), the desired conclusion follows once we have shown that $\dim K[P,Q]=|P|(|Q|-1)+1$ in the case that $P$ and $Q$ are connected.
Let $L$ be the quotient field of $K[P,Q]$. Since $K[P,Q]$ is an affine domain,  the $\dim K[P,Q] $ is the transcendence  degree of  $L/K$.

Let $L'$ be the field generated by the elements
\begin{enumerate}
\item[(i)] $u_q = \prod_{p\in P} x_{p,q}$ for  $q \in Q$,
\item[(ii)] $x_{p,q}/ x_{p,q'}$ for $p\in P$, $q ,q' \in Q$ and $q < q'$.
\end{enumerate}
 We will show that $L' =L$. It is obvious that elements $u_q$ belong to $L$. 
Let $I$ be any poset ideal of $P$ and let $q,q'\in Q$ with $q<q'$. Then $\phi_{I}^{(q,q')}: P \rightarrow Q$ defined by
\begin{equation*}
\phi_{I}^{(q,q')}(p)= \left\{ \begin{array}{ll}
         q, &  \text{if  $p \in I$}, \\
         q', &\text{otherwise},
                  \end{array} \right.
\end{equation*}
is an isotone map with image $\{q, q'\}$. Given any $p_0 \in P$, one can find two poset ideals $I,J$ of $P$ such that $J = I \cup \{p_0\}$ and $p_0 \notin I$. Then $u_{\phi_I^{(q,q')}}= \prod _{p \in I} x_{p,q}\prod_{p \notin I} x_{p,q'} $ and $u_{\phi_J^{(q,q')}}= (\prod _{p \in I}  x_{p,q})x_{p_0,q}( \prod_{p \notin I} x_{p,q'}) x_{p_0,q'}^{-1}$. It follows that $u_{\phi_J^{(q,q')}}/ u_{\phi_I^{(q,q')}} =  x_{p_0,q}/ x_{p_0,q'}$. It shows that all elements in (ii) belong to $L$. Therefore, $L' \subset L$.

In order to prove the converse inclusion, let $P=\{p_1, \ldots,p_n\}$ and let $q_1, \ldots, q_n$ be arbitrary elements in $Q$. We will show by induction on $r$ that  the monomial $x_{p_1,q_1} \cdots x_{p_r,q_1} x_{p_{r+1}, q_{r+1}} \cdots x_{p_n, q_n} $  can be obtained as a product  of  $x_{p_1,q_1} \cdots x_{p_n,q_n}$ and suitable elements in (ii) and their inverses. This will imply that any monomial generator of $K[P,Q]$ is contained in $L'$, because the element $x_{p_1,q_1}  x_{p_2, q_2} \cdots x_{p_{n}, q_{n}} $ can then be written as a product of $x_{p_1,q_1} \cdots  x_{p_n, q_1} $ and elements of type $\rm{(ii)}$ and their inverses. As a consequence  this  will imply that $L\subset L'$.

For $r=1$, the statement is trivial. Since $Q$ is connected, there exists a sequence of elements $\tilde{q}_1, \ldots, \tilde{q}_t $ in $Q$ with $\tilde{q}_1 = q_1$ and $\tilde {q}_t = q_{r+1}$ and $\tilde{q}_i$ and $\tilde{q}_{i+1}$ are comparable for all $i$ in $Q$ . Then $x_{p_{r+1}, q_1}/ x_{p_{r+1}, q_{r+1}} =  \prod_{i=1}^{t}( x_{p_{r+1}, \tilde{q}_i }/ x_{p_{r+1}, \tilde{q}_{i+1}})$, where each $x_{p_{r+1}, \tilde{q}_i }/ x_{p_{r+1}, \tilde{q}_{i+1}}$ or its inverse is a monomial of type (ii).
Then the monomial $x_{p_1,q_1} \cdots x_{p_r, q_1} x_{p_{r+1}, q_{1}} x_{p_{r+2}, q_{r+2}}\cdots x_{p_n, q_n} $ can be obtained as a  product of the monomial  $ x_{p_1,q_1} \cdots x_{p_r, q_1} x_{p_{r+1}, q_{r+1}} \cdots x_{p_n, q_n}$ and the monomial $x_{p_{r+1}, q_1}/ x_{p_{r+1}, q_{r+1}}$.

Let $T$ be a spanning tree (i.e.\ a maximal tree)  of $G(Q)$ and choose an element $p_0 \in P$. Next we  will show that $L$ is generated by the elements
\begin{enumerate}
\item[ (i)] $u_q = \prod_{p\in P} x_{p,q}$, for  $q \in Q$,
\item[(ii$'$)] $x_{p,q}/ x_{p,q'}$ for $p\in P \setminus \{p_0\}$ and $\{q, q'\} \in E( T)$ with $q<q'$,
\end{enumerate}
where $E(T)$ is the edge set of $T$. For any $q < q' $ in $Q$, we can obtain a sequence $\tilde{q_1}, \ldots, \tilde{q_t} $ in $T$ such that $\tilde{q_1} = q$ and $\tilde{q_t} = q'$ and $\tilde{q_i}$ and $\tilde{q}_{i+1}$ are neighbors for all $i = 1, \ldots, t$. Consequently, we obtain $x_{p,q}/x_{p,q'} = \prod_{i=1}^{t-1} x_{p,\tilde{q}_{i}}/x_{p,\tilde{q}_{i+1}}$. Also note that $x_{p_0, q}/ x_{p_0, q'} = (u_{q}/u_{q'}) (\prod_{p \in P\setminus \{p_0\}} (x_{p,q'}/x_{p,q}))$  for $q < q'$, and hence $x_{p_0, q}/ x_{p_0, q'}$ is a product of elements of type (i) and (ii$'$) and their inverses.   This shows that all monomials of type (ii) can be obtained as product of monomials of type (i) and type (ii$'$) and their inverses.

Let $\MA$ be  the set of  exponent vectors (in  $\QQ^{P\times Q}$) of the monomials  of type (i) and $\MB$ be the set of  exponent vectors  of the monomials  of type (ii$'$). We will show the set  of vectors $\MA\union\MB$ is  linearly independent. Then \cite[Proposition 7.1.17]{Vi} implies that the Krull dimension of $K[P, Q]$ is equal to the cardinality of the set $\MA\union \MB$.

 Note that  vectors in $\MA$  are linearly independent because their support is pairwise disjoint. Also, the vectors in $\MB$  are linearly independent. To see this,  we let $\MB_p \subset \MB$ be the set of exponent vectors of monomials  $x_{p,q}/ x_{p,q'}$ in (ii$'$) with $p$ fixed. Then $\MB$ is the disjoint union of the sets $\MB_p$ with $p\in P\setminus \{p_0\}$. Moreover, for distinct  $p,p'\in P\setminus\{p_0\}$  the vectors in $\MB_p$ and $\MB_{p'}$ have disjoint support. Thus it suffices to show that for a fixed $p\in P\setminus\{p_0\}$ the vectors in $\MB_p$ are linearly independent. Now we fix such $p\in P\setminus\{p_0\}$. Then the matrix formed by the vectors of  $\MB_p$  is the incidence matrix of the tree $T$ which is known to be of maximal rank, see for example \cite[Lemma 8.3.2]{Vi}. Thus the vectors of $\MB_p$ are linearly independent, as desired.

 Finally, to show that $\MA\union \MB$ is a set of linearly independent vectors it suffices to show $V\sect W=\{0\}$, where $V$ is $\QQ$-vector space spanned by   $\MA$ and $W$ is $\QQ$-vector space spanned by   $\MB$.  Thus we have to show that if $v\in V$ and $w\in W$ with $v=w$, then $v=w=0$. This is equivalent to say that if $u$ is a product of elements of (i) and its inverses and $u'$ is a product of elements of (ii$'$) and its inverses, and if $u=u'$, then $u=u'=1$.  This is indeed the case, because  if $u\neq 1$, then $u$ contains factors of the form $x_{p_0, q}$, while $u'$ does not.

Now we determine the cardinality of $\MA\union \MB$ (which coincides with $\dim K[P,Q]$), and  observe that  $|\MA|=|Q|$  and $|\MB|=(|P|-1)(|Q|-1)$, so that $|\MA\union \MB|=(|P|-1)(|Q|-1) + |Q|=|P|(|Q|-1)+1$.
 \end{proof}

\begin{Corollary}
\label{birational}
Let $P$ and $Q$ be finite connected posets with $|P|=n$ and $|Q|=m$, and let $X_{P,Q}$ be the irreducible variety given by the defining ideal  $J_{P,Q }$ of $K[P,Q]$. Then $X_{P,Q}$ is birationally equivalent to the variety $Y_{n,m}$ whose coordinate ring is the $n$-fold Segre product of the $m$-dimensional polynomial ring over $K$.
\end{Corollary}

\begin{proof} Let $P=\{p_1,\ldots,p_n\}$, and let $L$ be the quotient field of $K[P,Q]$ and $T$ be the toric ring  whose generators are of the form $\prod_{j=1}^nx_{p_j,q_{j}}$ with  $ q_{j}\in Q$. Note that
$$T=S_1*S_2*\cdots *S_n,$$
where $S_i=K[x_{p_i,q}\: q\in Q]$ for $i=1,\ldots,n$. It was shown in the proof of Theorem~\ref{dimension} that  $L$ is also the quotient field of $T$.
This yields the desired conclusion.
\end{proof}

\section{Normality of $K[P,Q]$}

In this section we prove  normality of $K[P,Q]$  in certain cases. As $K[P,Q]$ is a toric ring, normality of $K[P,Q]$, by a theorem of Hochster \cite[Theorem 1]{Ho},  implies   that $K[P,Q]$ is Cohen--Macaulay as well.

Let $H$ be an affine semigroup and $\ZZ H$ be the associated group of $H$. The semigroup $H$ is normal, if  whenever $da\in H$ for $a\in \ZZ H$  and some $d\in\NN$, then $a\in H$. By \cite[Theorem 6.1.4]{BH}, $K[H]$ is normal if and only if $H$ is normal. We apply this criterion to  $K[P,Q]$. Since $K[P,Q]$ is a toric ring,  there exists an   affine semigroup $H$   such that $K[H]=K[P,Q]$.  It follows that $K[P,Q]$ is normal if and only if $H$ is normal.  In our particular situation, when $P=\{p_1,\ldots,p_n\}$,  the monomials in  $K[\ZZ H]$ corresponding to the elements in $\ZZ H$ are of the form
$u_1^{\pm 1}u_2^{\pm 1}  \cdots u_s^{\pm 1}$ with $u_i=x_{p_1,q_{i1}}x_{p_2,q_{i2}}\cdots x_{p_n,q_{in}}$ for $i=1,\ldots,s$,  $s\geq 1$ and $q_{ij}\in Q$.  Thus it will follow that $K[P,Q]$ is normal if whenever $(u_1^{\pm 1}u_2^{\pm 1}  \cdots u_s^{\pm 1})^d\in K[P,Q]$, then $u_1^{\pm 1}u_2^{\pm 1}  \cdots u_s^{\pm 1}\in K[P,Q]$.

\begin{Lemma}\label{restriction}
Let $P,Q$ be two finite posets with $P=\{p_1,\ldots,p_n\}$,    and let as before $S_i=K[x_{p_i,q}\: q\in Q]$ for $i=1,\ldots,n$. Assume that for any integer $d>1$ and any sequence of monomials $v_1,\ldots, v_t\in S_1*S_2*\cdots *S_n$ with   $(v_1v_2  \cdots v_t)^d\in K[P,Q]$ it follows that  $v_1v_2  \cdots v_t\in K[P,Q]$. Then the toric ring $K[P,Q]$ is normal.
\end{Lemma}

\begin{proof}
	Suppose $(u_1^{\pm 1}u_2^{\pm 1}\cdots u_s^{\pm 1})^d\in K[P,Q]$. We show that there exist monomials  $v_1,\ldots, v_t\in S_1*S_2*\cdots *S_n$ such that  $(u_1^{\pm 1}u_2^{\pm 1}\cdots u_s^{\pm 1})^d=(v_1\cdots v_t)^d$. Then our assumption implies that  $v_1\cdots v_t\in K[P,Q]$. Since  $(u_1^{\pm 1}u_2^{\pm 1}\cdots u_s^{\pm 1})^d=(v_1\cdots v_t)^d$ and since $K[P,Q]$ is a toric ring, we have $u_1^{\pm 1}u_2^{\pm 1}\cdots u_s^{\pm 1}=v_1\cdots v_t$, and so the desired conclusion follows.
	
	Without loss of generality we may assume that
	$$(u_1^{\pm 1}u_2^{\pm 1}\cdots u_s^{\pm 1})^d=(u_1^{-1}\cdots u_r^{-1})^d(u_{r+1}\cdots u_s)^d$$
	 for some $r\leq s$. Let $u_i=\prod_{j=1}^{n}x_{p_j,q_{ij}}$ for $1\leq i\leq s$. Then
	 \[
	 (u_1^{-1}\cdots u_r^{-1})^d(u_{r+1}\cdots u_s)^d=\prod_{i=1}^{n}[(\prod_{k=1}^{r}x_{p_i,q_{ki}}^{-d})(\prod_{k=r+1}^{s}x_{p_i,q_{ki}}^{d})]
	 \]
	belongs to $K[P,Q]$. Since the elements in $K[P,Q]$ have no negative powers, it follows that each factor in $\prod_{k=1}^{r}x_{p_i,q_{ki}}^{-d}$ cancels against a factor in $\prod_{k=r+1}^{s}x_{p_i,q_{ki}}^{d}$. Without loss of generality we may assume that for each $i$, $\prod_{k=1}^{r}x_{p_i,q_{ki}}^{-d}$ cancels against $\prod_{k=r+1}^{2r}x_{p_i,q_{ki}}^{d}$. Then
	\[
	(u_1^{\pm 1}u_2^{\pm 1}\cdots u_s^{\pm 1})^d=
	\prod_{i=1}^{n}(\prod_{k=2r+1}^{s}x_{p_i,q_{ki}}^{d})=(v_1\cdots v_{s-2r})^d,
	\]
	where $v_j=\prod_{i=1}^{n}x_{p_i,q_{2r+j,i}}$.
\end{proof}

Based on the criterion given in Lemma~\ref{restriction} we show

\begin{Theorem}\label{normaltheorem}
Let $P$ and $Q$ be finite posets, and   let $P'$ be a poset which is obtained from $P$ by adding an element $p'$ to $P$  with the property that $p'$ has a unique neighbor in $G(P')$.  Then $K[P',Q]$ is normal if and only if $K[P,Q]$ is normal.
\end{Theorem}

\begin{proof}
By  Corollary~\ref{normalconnected} we may assume that  $P$ is connected. Let  $P = \{p_1 ,\ldots, p_n\}$ and set $p_{n+1}=p'$. We may assume that $p_n$ is the unique  neighbor of $p_{n+1}$ in $G(P)$ and that $p_n<p_{n+1}$. In other words, $p_n$ is the unique element  in $P$ covered by $p_{n+1}$.

Suppose first that $K[P',Q]$ is normal but $K[P,Q]$ is  not normal. Then  there exist monomials $u_1,\ldots,u_s \in S_1*\cdots *S_n$
such that $(u_1\cdots u_s)^d \in K[P,Q]$ but $u_1\cdots u_s \not\in K[P,Q]$.

We write $\varphi_i$  for the (not necessarily isotone) map $P \to Q$
corresponding to $u_i$, and define the map $\varphi'_i : P' \to  Q$
 by setting
 \[
 \varphi'_i(p)
 =\left\{
  \begin{array}{ll}
    \varphi_i(p), & \hbox{if $p\neq p_{n+1}$}, \\
     \varphi_i(p_n),  & \hbox{if $p= p_{n+1}$}.
  \end{array}
\right.
\]
Let $u_i'$ be the monomial corresponding to $\phi'_i$. Then we have
$(u'_1\cdots u'_s)^d \in K[P',Q]$ but   $u'_1\cdots u'_s \not \in  K[P',Q]$, a contradiction.

Conversely,  suppose that $K[P,Q]$ is normal.
Let $u_1,\ldots,u_s\in S_1*S_1*\cdots * S_{n+1}$ with $u_i=\prod_{j=1}^{n+1}x_{p_j,q_{ij}}$ for $i=1,\ldots,s$. By Lemma~\ref{restriction} it is enough to show that if $(u_1u_2\cdots u_s)^d\in K[P',Q]$ for some $d\in \NN$, then $u_1u_2\cdots u_s\in K[P',Q]$.

In order  to prove this, we first observe that $\varphi:\ P'\to Q$ is isotone, if and only if $\varphi(p)\leq \varphi(q)$ for all pairs $p, q\in P'$ for which   $q$ covers $p$. Hence, $\varphi:\ P'\to Q$ is isotone if and only if the  restriction  of $\varphi$ to $P$ is isotone and $\varphi(p_n)\leq \varphi(p_{n+1})$.

As a consequence of this observation we obtain the following statement:  let $v_1,\ldots,v_t\in S_1*S_1*\cdots * S_{n+1}$ with $v_i= \prod_{j=1}^{n+1}x_{p_j,q'_{ij}}$ for $i=1,\ldots,t$, and set $\tilde{v}_i= \prod_{j=1}^{n}x_{p_j,q'_{ij}}$ for $i=1,\ldots,t$.
Then
 $v_1\cdots v_t\in K[P',Q]$
if and only if
\begin{enumerate}
\item[(i)] $\tilde{v}_1\cdots \tilde{v}_t\in K[P,Q]$,

\item[(ii)] there exists a permutation $\pi\: [t]\to [t]$ such that $q'_{in}\leq q'_{\pi(i),n+1}$ for all $i$.
\end{enumerate}

\medskip
Condition (ii) can be rephrased as follows: let $G$ be the bipartite graph with bipartition $V(G)=V_1\union V_2$, where $V_1=\{x_1,\ldots, x_t\}$ and $V_2=\{y_1,\ldots,y_t\}$. We let  $\{x_i,y_j\}$ be  an edge of $G$ if and only if $q'_{in}\leq q'_{j,n+1}$. We call $G$ the graph attached to $v_1,\ldots,v_t$. By definition, a perfect matching of $G$ is a disjoint union of $t$ edges of $G$.

Now condition (ii) is equivalent to the following condition:
\begin{enumerate}
\item[(ii$'$)] The graph $G$ attached to $v_1,\ldots,v_t$ admits a perfect matching.
\end{enumerate}

Now we come back to the sequence $u_1,\ldots,u_s\in S_1*S_1*\cdots * S_{n+1}$ with $u_i=\prod_{j=1}^{n+1}x_{p_j,q_{ij}}$ for $i=1,\ldots,s$ for which $(u_1u_2\cdots u_s)^d\in K[P',Q]$.

Then (i) implies that $(\tilde{u}_1\cdots \tilde{u}_s)^d\in K[P,Q]$. By assumption, $K[P,Q]$ is normal. This implies that $\tilde{u}_1\cdots \tilde{u}_s\in K[P,Q]$, where $\tilde{u}_i=\prod_{j=1}^{n}x_{p_j,q_{ij}}$. Thus it will follow that $u_1\cdots u_s\in K[P',Q]$ if the graph $G$ attached to $u_1,\ldots, u_s$ admits a perfect matching.

We let $G^{(d)}$ be the graph attached to $d$ copies of $u_1,\ldots,u_s$:
\[
\underbrace{u_1,\ldots,u_s,\  \ldots\ , u_1,\ldots,u_s}_{d\  \text{times}}.
\]
Since $(u_1u_2\cdots u_s)^d\in K[P,Q]$ it follows from (ii$'$) that  $G^{(d)}$ admits
a perfect matching.

 Note that $G^{(d)}$ may be viewed as the bipartite graph with vertex decomposition $V(G^{(d)}) =V_1^{(d)}\union V_2^{(d)}$, where $V_1^{(d)}=\{x_i^{(k)}\:\;  i=1,\ldots,s, k=1,\ldots,d\}$ and $V_2^{(d)}=\{y_i^{(k)}\: i=1,\ldots,s, k=1,\ldots,d\}$.  The edges of  $G^{(d)}$ are $\{x_i^{(k)},y_j^{(l)}\}$ with $q_{in}\leq q_{j,n+1}$ and $k,l\in [d]$.

Suppose that $G$ has no perfect matching. Then Hall's marriage theorem (see for example \cite[Lemma9.1.2]{HHBook}) implies that there exists a set $S\subset V_1$ such that $|N_G(S)|<S$. Here
\[
N_G(S)=\{y_j\:\; \{x_i,y_j\} \text{ with } x_i\in S\}
\]
is the set of neighbors of $S$.

Let $S^{(d)} = \{x_i^{(k)} \:\;  x_i \in S,  1\leq k\leq d\}$.
Then  $S^{(d)} \subset V_1^{(d)}$ and
\[
|N_{G^{(d)}}(S^{(d)})| = d|N_G(S)|<d|S|=|S^{(d)}|.
\]
Thus,  $G^{(d)}$ does not have a perfect matching, a contradiction.
\end{proof}

By using Theorem~\ref{normaltheorem} and an obvious induction argument, we obtain

\begin{Corollary}
\label{normal}
Let $P$ and  $Q$ be finite posets. Then $K[P,Q]$ is a normal Cohen--Macaulay domain if $G(P)$ is a tree.
\end{Corollary}

Corollary~\ref{normal} implies in particular that $K[[n] , Q]$ is normal. On the other hand, the normality of $K[P,[n]]$ has been proved in \cite[Corollary 4.2]{EHM}.  These types of algebras are called {\em letterplace algebras}. Thus we have

\begin{Corollary}
\label{letterplace}
All letterplace algebras are normal Cohen--Macaulay domains.
\end{Corollary}

It seems that $K[P,Q]$ is normal, not only when $G(P)$ is a tree. For example,  it can be shown by using Normaliz \cite{N} that   the isotonian algebra  $K[Q,Q]$ is normal for the poset $Q$ shown  in Figure~\ref{salam}.

Further computational evidence  and the above  special cases lead us to the following

\begin{Conjecture}
\label{webelieveit}
Let $P$ and $Q$ be finite posets. Then $K[P,Q]$ is a normal Cohen--Macaulay domain.
\end{Conjecture}

\section{Gr\"obner bases}

Let $P$ and $Q$ be finite posets.
Let $S=K[y_{\varphi}\: \; \varphi \in \Hom(P,Q)]$ be the polynomial ring
in the variables $y_{\varphi}$, and let $\pi : S \rightarrow K[P,Q]$
be the $K$-algebra homomorphism defined by $y_{\varphi} \mapsto \prod_{p\in P} x_{p,\varphi(p)}$.
We denote the kernel of  $\pi$ by $J_{P,Q}$.

Let  $P=\{p_1,p_2,\ldots,p_n\}$. We may assume that the labeling of the elements of $P$ is chosen such that $p_i<p_j$ implies $i<j$. Similarly, $Q=\{q_1,q_2,\ldots,q_m\}$ is labeled. Having fixed this labeling we  sometimes  write $y_{(j_1,j_2,\ldots,j_n)}$ for $y_\varphi$  if $\varphi(p_i)=q_{j_i}$ for $i=1,\ldots,n$.

For $Q=[2]$ the ideal $J_{P,Q}$   is the defining ideal of the Hibi ring associated with  the distributive lattice $L=\Hom(P,[2])$. In this case it is known that $J_{P,Q}$ has a quadratic Gr\"obner basis, see \cite[Theorem 10.1.3]{HHBook}. For arbitrary posets $P$ and $Q$ the ideal $J_{P,Q}$ is not always generated in degree $2$.  The simplest example is given by the posets $P$ and $Q$ displayed in Figure~\ref{salam}.
\begin{figure}[!htbp]
	\begin{center}
\psset{xunit=.8cm,yunit=.8cm,algebraic=true,dimen=middle,dotstyle=o,dotsize=5pt 0,linewidth=0.8pt,arrowsize=3pt 2,arrowinset=0.25}
\begin{pspicture*}(1.,-1.5)(13.,4.)
\psline(3.04,3.14)(3.06,0.42)
\psline(6.92,3.06)(6.92,0.58)
\psline(9.1,3.16)(6.92,0.58)
\psline(9.1,0.6)(9.1,3.16)
\psline(9.1,0.6)(11.4,3.2)
\psline(11.34,0.54)(11.4,3.2)
\rput[tl](2.94,-0.84){$P$}
\rput[tl](8.9,-0.96){$Q$}
\rput[tl](2.94,3.7){\begin{scriptsize}2\end{scriptsize}}
\rput[tl](6.74,0.16){\begin{scriptsize}1\end{scriptsize}}
\rput[tl](8.96,0.2){\begin{scriptsize}2\end{scriptsize}}
\rput[tl](11.24,0.14){\begin{scriptsize}3\end{scriptsize}}
\rput[tl](6.76,3.58){\begin{scriptsize}4\end{scriptsize}}
\rput[tl](8.96,3.66){\begin{scriptsize}5\end{scriptsize}}
\rput[tl](11.24,3.68){\begin{scriptsize}6\end{scriptsize}}
\rput[tl](2.94,0.12){\begin{scriptsize}1\end{scriptsize}}
\psline(11.34,0.54)(6.92,3.06)
\begin{scriptsize}
\psdots[dotstyle=*,linecolor=black](3.04,3.14)
\psdots[dotstyle=*,linecolor=black](3.06,0.42)
\psdots[dotstyle=*,linecolor=black](6.92,3.06)
\psdots[dotstyle=*,linecolor=black](6.92,0.58)
\psdots[dotstyle=*,linecolor=black](9.1,3.16)
\psdots[dotstyle=*,linecolor=black](9.1,0.6)
\psdots[dotstyle=*,linecolor=black](11.4,3.2)
\psdots[dotstyle=*,linecolor=black](11.34,0.54)
\end{scriptsize}
\end{pspicture*}
\caption{$J_{P,Q}$ is not quadratic}\label{salam}
\end{center}
\end{figure}
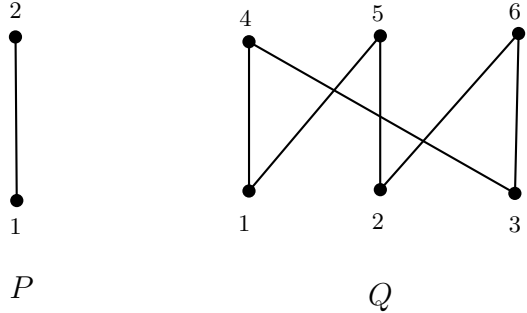

In this example $J_{P,Q}$ is generated by the binomial $y_{(1,4)}y_{(2,5)}y_{(3,6)}-y_{(2,4)}y_{(3,5)}y_{(1,6)}$.

\medskip
It is known
that in general $J_{P,Q}$ is generated by binomials.
Consult \cite{dojoEN} for fundamental materials on
toric ideals and Gr\"obner bases.

We identify each $\phi \in \Hom(P,Q)$ with the sequence
$(j^{(\phi)}_{1}, \ldots, j^{(\phi)}_{n})$, where
$\phi(p_{i}) = q_{j^{(\phi)}_{i}}$ for $1 \leq i \leq n$.
We introduce the total ordering $<$ of the variables
$y_{\phi}$ with $\phi \in \Hom(P,Q)$
by setting $y_{\phi} < y_{\psi}$ if $j^{(\phi)}_{i_{0}} < j^{(\psi)}_{i_{0}}$,
where $i_{0}$ is the smallest integer
for which $j^{(\phi)}_{i_{0}} \neq j^{(\psi)}_{i_{0}}$.
Let $<_{\rm rev}$ denote the reverse lexicographic order
on $S$ induced by the above ordering $<$
of the variables $y_{\phi}$ with $\phi \in \Hom(P,Q)$.

\begin{Example}
\label{sequence}
{\em
Let $P = \{ p_{1}, p_{2}, p_{3} \}$, where $p_{1} < p_{2}$ and $p_{1} < p_{3}$,
and $Q = \{ q_{1}, q_{2}, q_{3} \}$, where $q_{1} < q_{3}$ and $q_{2} < q_{3}$.
Then the total ordering $<$ on the variables
$y_{\phi}$ with $\phi \in \Hom(P,Q)$ is
\[
y_{(1,1,1)}<y_{(1,1,3)}<y_{(1,3,3)}<y_{(2,2,2)}<y_{(2,2,3)}<y_{(2,3,3)}<y_{(3,3,3)}.
\]
}
\end{Example}


We say that $u_{\phi}u_{\psi}$ is {\em nonstandard} with respect to $<_{\rev}$
if there exist $\phi'$ and $\psi'$ for which $u_{\phi}u_{\psi} = u_{\phi'}u_{\psi'}$
and $y_{\phi'}y_{\psi'} <_{\rev} y_{\phi}y_{\psi}$.  An expression
$w = u_{\phi_{1}}u_{\phi_{2}}\cdots u_{\phi_{s}}$
of a monomial $w$ belonging to $K[P,Q]$ is called {\em standard}
if no $y_{\phi_{i}}y_{\phi_{j}}$, where $1 \leq i < j \leq s$, is nonstandard.
It follows that each monomial possesses a standard expression.
However, a standard expression of a monomial may not be unique.

\begin{Example}
\label{cycle}
Let $P = [2]$ and $Q$ be as in Figure~\ref{salam}.
Then each of the expressions
$u_{(1,4)}u_{(2,5)}u_{(3,6)}$ and $u_{(1,5)}u_{(2,6)}u_{(3,4)}$
of the monomial
$x_{11}x_{12}x_{13}x_{24}x_{25}x_{26}$
is standard.
\end{Example}

\begin{Lemma}
\label{basic}
Let $P = \{p_{1}, \ldots, p_{n}\}$ be an arbitrary finite poset and $Q = [2]$.
Then every monomial belonging to $K[P,Q]$ possesses a unique standard expression.
\end{Lemma}

\begin{proof}
Let $w = u_{\phi_{1}} \cdots u_{\phi_{s}}$
with $y_{\phi_{1}} \leq \cdots \leq y_{\phi_{s}}$.
Suppose that there exist $2 \leq i_{0} \leq n$ and $1 \leq k < k' \leq s$ for which
$\phi_{k}(p_{i_{0}}) = 2$ and $\phi_{k'}(p_{i_{0}}) = 1$.
Since $y_{\phi_{1}} \leq \cdots \leq y_{\phi_{s}}$,
it follows that there is $1 \leq i' < i_0$ with
$\phi_{k}(p_{i'}) = 1$ and $\phi_{k'}(p_{i'}) = 2$.
Since each of the inverse images $\phi_{k}^{-1}(1)$ and $\phi_{k'}^{-1}(1)$
is a poset ideal of $P$, it follows that each of the maps
$\psi_{k} : P \to Q$ and $\psi_{k'} : P \to Q$ defined by setting
$\psi_{k}(p_{i}) = \min\{\phi_{k}(p_{i}), \phi_{k'}(p_{i})\}$ and
$\psi_{k'}(p_{i}) = \max\{\phi_{k}(p_{i}), \phi_{k'}(p_{i})\}$
for $1 \leq i \leq n$ is isotone.  Furthermore, one has
$y_{\psi_{k}}y_{\psi_{k'}} <_{\rev} y_{\phi_{k}}y_{\phi_{k'}}$ and
$u_{\phi_{k}}u_{\phi_{k'}} = u_{\psi_{k}}u_{\psi_{k'}}$.
Thus $w = u_{\phi_{1}} \cdots u_{\phi_{s}}$ cannot be standard.
Since $y_{\phi_{1}} \leq \cdots \leq y_{\phi_{s}}$,
it follows that
\[
\phi_{1}(p_{1}) \leq \phi_{2}(p_{1}) \leq \cdots \leq \phi_{s}(p_{1}).
\]
Hence, if $w = u_{\phi_{1}} \cdots u_{\phi_{s}}$
is standard with $y_{\phi_{1}} \leq \cdots \leq y_{\phi_{s}}$, then
\[
\phi_{1}(p_{i}) \leq \phi_{2}(p_{i}) \leq \cdots \leq \phi_{s}(p_{i})
\]
for $1 \leq i \leq n$.  This guarantees that a standard expression
of each monomial belonging to $K[P,Q]$ is unique, as desired.
\end{proof}

A finite poset $Q$ is called a {\em rooted tree} if whenever $\alpha, \beta$
and $\gamma$ belong to $Q$ with $\beta < \alpha$ and $\gamma < \alpha$,
then either $\beta \leq \gamma$ or $\gamma \leq \beta$.  In other words,
a finite poset $Q$ is a rooted tree if a maximal chain of $Q$
descending from each $\alpha \in Q$ is unique.
A finite poset is a {\em co-rooted tree} if its  dual poset is a rooted tree.

\begin{Lemma}
\label{basicbasic}
Let $P = [2]$ and
$Q$ 
a co-rooted tree.
Then each monomial belonging to $K[P,Q]$ possesses a unique standard expression.
\end{Lemma}

\begin{proof}
%
Let $u$ be a monomial belonging to $K[P,Q]$.  Let
$u = u_{\phi_{1}} \cdots u_{\phi_{s}}$ and
$u = u_{\psi_{1}} \cdots u_{\psi_{s}}$ be standard expressions of
$u$ with $y_{\phi_{1}} \leq \cdots \leq y_{\phi_{s}}$
and $y_{\psi_{1}} \leq \cdots \leq y_{\psi_{s}}$.
Then one has $\phi_{k}(1) = \psi_{k}(1)$
for $1 \leq k \leq s$.  Furthermore,
$\phi_{k}(1) \leq \phi_{k}(2)$ and $\psi_{k}(1) \leq \psi_{k}(2)$
for $1 \leq k \leq s$.
In order to show that the standard expressions
$u_{\phi_{1}} \cdots u_{\phi_{s}}$ and
$u_{\psi_{1}} \cdots u_{\psi_{s}}$ coincide,
we may assume  without loss of generality that $\phi_{i} \neq \psi_{j}$ for all $i$ and $j$.
Since $Q$ is a co-rooted tree and since
$\phi_{1}(1) \leq \phi_{1}(2)$ and $\psi_{1}(1) \leq \psi_{1}(2)$
with $\phi_{1}(1) = \psi_{1}(1)$,
it follows that $\phi_{1}(2)$ and $\psi_{1}(2)$ are comparable.
Let, say, $\phi_{1}(2) < \psi_{1}(2)$.  Since
$u_{\phi_{1}} \cdots u_{\phi_{s}} = u_{\psi_{1}} \cdots u_{\psi_{s}}$,
there is $2 \leq k_{0} \leq s$
with $\psi_{k_{0}}(2) = \phi_{1}(2)$.
Hence $\psi_{k_{0}}(1) \leq \psi_{k_{0}}(2) < \psi_{1}(2)$
and $\psi_{1}(1) \leq \psi_{k_{0}}(2) < \psi_{1}(2)$.
One can then define $\psi'_{1}$ and $\psi'_{k_{0}}$ belonging to $\Hom(P,Q)$
by setting
\[
\psi'_{1}(1) = \psi_{1}(1), \, \, \, \psi'_{1}(2) = \psi_{k_{0}}(2), \, \, \,
\psi'_{k_{0}}(1) = \psi_{k_{0}}(1), \, \, \, \psi'_{k_{0}}(2) = \psi_{1}(2).
\]
Then $u_{\psi_{1}}u_{\psi_{k_{0}}} = u_{\psi'_{1}}u_{\psi'_{k_{0}}}$ with
$y_{\psi'_{1}} <y_{\psi_{1}} \leq y_{\psi_{k_{0}}}$.
Furthermore, $y_{\psi'_{k_0}} \geq y_{\psi'_{1}}$.  Hence
$y_{\psi'_{k_0}}y_{\psi'_{1}}<_{\rev}y_{\psi_{k_0}}y_{\psi_{1}}$.
It then follows that $y_{\psi_{1}}y_{\psi_{k_{0}}}$ cannot be standard.
Hence
a standard expression of each monomial belonging to $K[P,Q]$ is unique,
as desired.
\end{proof}

\begin{Theorem}
\label{Boston}
Let $P$ be a chain  and
$Q$ 
a co-rooted tree.
Then each monomial belonging to $K[P,Q]$ possesses a unique standard expression.
\end{Theorem}

\begin{proof}
%
We may assume that  $P = [n]$. Then  for each $\phi \in \Hom(P,Q)$,
one has $q_{j^{(\phi)}_{1}} \leq \cdots \leq q_{j^{(\phi)}_{n}}$.
In other words, the image $\phi([n])$ is a multichain (chain with repetitions) of $Q$ of length $n$.
It then follows that $u_{\phi}u_{\psi}$ with $y_{\phi} < y_{\psi}$
is nonstandard if and only if there is $2 \leq i \leq n$ with
$\phi(i) > \psi(i)$ such that
$\phi(i-1) \leq \psi(i)$ and $\psi(i-1) \leq \phi(i)$.
In fact, one has $y_{\psi'}y_{\phi'} <_{\rev} y_{\psi}y_{\phi}$, where
$\phi' = (j_{\phi}^{(1)}, \ldots, j_{\phi}^{(i-1)}, j_{\psi}^{(i)}, \ldots, j_{\psi}^{(n)})$
and
$\psi' = (j_{\psi}^{(1)}, \ldots, j_{\psi}^{(i-1)}, j_{\phi}^{(i)}, \ldots, j_{\phi}^{(n)})$.

Given $\phi \in \Hom(P,Q)$, we introduce $\phi^{*} \in \Hom(P \setminus \{ n \},Q)$
by setting $\phi^{*}(i) = \phi(i)$ for $i \in [n-1]$.
Let $u$ be a monomial belonging to $K[P,Q]$.  Let
$u = u_{\phi_{1}} \cdots u_{\phi_{s}}$ and
$u = u_{\psi_{1}} \cdots u_{\psi_{s}}$ be standard expressions of
$u$ with $y_{\phi_{1}} \leq \cdots \leq y_{\phi_{s}}$
and $y_{\psi_{1}} \leq \cdots \leq y_{\psi_{s}}$.
The above observation guarantees that each of
$u_{\phi^{*}_{1}} \cdots u_{\phi^{*}_{s}}$ and
$u_{\psi^{*}_{1}} \cdots u_{\psi^{*}_{s}}$ is a standard expression.
Thus, working on induction on $n$, it follows that
$\phi^{*}_{k} = \psi^{*}_{k}$ for $1 \leq k \leq s$.
In particular $\phi_{k}(n-1) = \psi_{k}(n-1)$
for $1 \leq k \leq s$.

Now, in order to show that the standard expressions
$u_{\phi_{1}} \cdots u_{\phi_{s}}$ and
$u_{\psi_{1}} \cdots u_{\psi_{s}}$ coincide,
one must show that $\phi_{k}(n) = \psi_{k}(n)$ for $1 \leq k \leq s$.
This can be done by using the same technique as in the proof of
Lemma \ref{basicbasic}.
\end{proof}

\begin{Corollary}
\label{Sydney}
Let $P$ be a chain and suppose that each connected component of $Q$ is
 either a rooted or a co-rooted poset.
Then the toric ideal $J_{P,Q}$ possesses a quadratic Gr\"obner basis.
\end{Corollary}

\begin{proof}
Let $Q$ be a co-rooted tree.
Let $\MG$ denote the set of quadratic binomials of $S$ of the form
$y_{\phi}y_{\psi} - y_{\phi'}y_{\psi'}$ with
$y_{\phi'}y_{\psi'} <_{\rev} y_{\phi}y_{\psi}$.
Furthermore, write ${\rm in}_{<_{\rev}}(\MG)$ for the monomial ideal of
$S$ which is generated by those quadratic monomials
$y_{\phi}y_{\psi}$ which are nonstandard.
Let $w$ and $w'$ be monomials of $S$
with $w \neq w'$ such that neither $w$ nor $w'$ belongs to
${\rm in}_{<_{\rev}}(\MG)$.
It then follows from Theorem \ref{Boston} that
$\pi(w) \neq \pi(w')$.
In other words, the set of those monomials $\pi(w)$ of $K[P,Q]$
with $w \not\in {\rm in}_{<_{\rev}}(\MG)$ is linearly independent.
By virtue of the well-known technique
(\cite[Lemma 1.1]{AHH} and \cite[(0.1)]{OhHirootsystem}) on initial ideals,
this fact guarantees that $\MG$ is a quadratic Gr\"obner basis of $J_{P,Q}$
with respect to $<_{\rev}$, as desired.

Let $Q$ be a rooted tree.
Lemma~\ref{obvious} says that $K[P^{\vee}, Q^{\vee}]$ is isomorphic to $K[P,Q]$.
Since $P^{\vee} = P$ and since $Q^{\vee}$ is a co-rooted tree, it follows that
$J_{P,Q}$ possesses a quadratic Gr\"obner basis, as required.

Now assume that each component $Q_i$ $(i=1,\ldots,m)$ of $Q$ is either rooted or co-rooted. As seen before, the toric ideal of each $K[P,Q_i]$ has a quadratic Gr\"obner basis. Since  by Lemma~\ref{segre}(b),  $K[P,Q]$ is the tensor product of the $K$-algebras $K[P,Q_i]$ it follows that the Gr\"obner basis of $J_{P, Q}$ is the union of the Gr\"obner bases of the $J_{P, Q_i}$. Thus the desired result follows.
\end{proof}

As we have seen at the beginning  of this section, $J_{P,Q}$ is not always generated by quadratic binomials. We say that  $C$ is a {\em poset cycle} of length $\ell$, if the vertices of $C$ are $a_1,\ldots,a_\ell, b_1,\ldots,b_\ell$ whose covering relations are $a_i<b_i$, $a_i<b_{i+1}$ for $1\leq i\leq \ell$ where $b_{\ell+1}=b_1$. The poset $Q$ in Figure~\ref{salam} is a poset $6$-cycle. Note, that not any cycle of $G(P)$ is a poset cycle, see for example Figure~\ref{chetori}.
\begin{figure}[!htbp]
	\begin{center}
\psset{xunit=.8cm,yunit=.8cm,algebraic=true,dimen=middle,dotstyle=o,dotsize=5pt 0,linewidth=0.8pt,arrowsize=3pt 2,arrowinset=0.25}
\begin{pspicture*}(3.,-2.6)(11.,3.)
\psline(7.06,2.34)(5.7,0.94)
\psline(5.72,-1.04)(5.7,0.94)
\psline(7.26,-2.3)(5.72,-1.04)
\psline(8.68,-0.08)(7.06,2.34)
\psline(8.68,-0.08)(7.26,-2.3)
\begin{scriptsize}
\psdots[dotstyle=*,linecolor=black](7.06,2.34)
\psdots[dotstyle=*,linecolor=black](5.7,0.94)
\psdots[dotstyle=*,linecolor=black](5.72,-1.04)
\psdots[dotstyle=*,linecolor=black](7.26,-2.3)
\psdots[dotstyle=*,linecolor=black](8.68,-0.08)
\end{scriptsize}
\end{pspicture*}
\caption{A cycle but not a poset cycle}\label{chetori}
\end{center}
\end{figure}
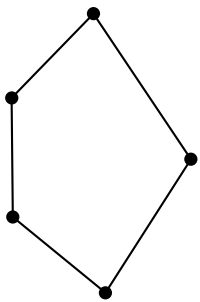

Based on Theorem~\ref{Boston},  Corollary~\ref{Sydney} and experimental evidence we propose

\begin{Conjecture}
\label{Berkeley}
Let $P,Q\in \MP$, and assume that $Q$ does not contain any induced poset cycle of length greater than 4. Then $J_{P,Q}$ is generated by quadratic binomials.
\end{Conjecture}

The poset $Q$ given in Figure~\ref{salam} is  a poset cycle. As expected by our conjecture, $J_{[2],Q}$ is not generated by quadrics. Actually, $J_{[2],Q}$ does not even possess a quadratic Gr\"obner basis with respect to the reverse lexicographic term order induced by the natural order of the variables.


\begin{thebibliography}{}

\bibitem{AHH} A.~Aramova, J.~Herzog and T.~Hibi, Finite Lattices and Lexicographic Gr\"obner Bases, European~J.~Combin. {\bf 21} (2000), 431--439.


\bibitem{BH} W.~Bruns and J.~Herzog, Cohen--Macaulay rings (Revised Edition), Cambridge Studies in advanced mathematics {\bf 39}, Cambridge University Press, 1996.

%
%
%
\bibitem{dojoEN}
T.~Hibi, Ed., Gr\"{o}bner Bases: Statistics and Software Systems, Springer, 2013.
%
%
%

\bibitem{FGH} G.~Fl{\o}ystad, B.M.~Greve and J.~Herzog, Letterplace and co-letterplace ideals of posets, \href{http://arxiv.org/abs/1501.04523v1}{arXiv:1501.04523}


\bibitem{GW} S.~Goto and  K.~Watanabe, On graded rings, I, J.~Math.~Soc.~Japan
{\bf 30} (1978), 179--213.


\bibitem{EHM}
V.~Ene, J.~Herzog and F.~Mohammadi,
Monomial ideals and toric rings of Hibi type arising from a finite poset,
 European J. Combin. {\bf 32} (2011), 404--421.


\bibitem{HHBook}
J.~Herzog and T.~Hibi, Monomial ideals,  Graduate Texts in Mathematics {\bf 260},
Springer, London, 2010.

\bibitem{HH} J.~Herzog and  T.~Hibi, Distributive lattices, bipartite graphs and Alexander duality,  J. Algebraic Combin. {\bf 22} (2005), 289--302.

\bibitem{HHBook} J.~Herzog and  T.~Hibi, Monomial Ideals, Graduate Text in Mathematics, Springer, 2011.

\bibitem{HSQ} J.~Herzog, A.~A.~Qureshi and A.~Shikama,  Alexander duality for monomial ideals associated with isotone maps between posets, to appear in J. Algebra and its Applications.

\bibitem{Hi} T.~Hibi, Distributive lattices, affine semigroup rings and algebras with straightening
laws, Commutative Algebra and Combinatorics (M. Nagata and H. Matsumura, Eds.)
Adv. Stud. Pure Math. {\bf 11}, North-Holland, Amsterdam, 1987, 93--109.


\bibitem{Ho} M.~Hochster, Ring of invariants of tori, Cohen--Macaulay rings generated by monomials, and polytopes, Ann.~Math. {\bf 96} (1972), 318--337.

\bibitem{KKM} M.~Juhnke-Kubitzke, L.~Katth\"an and  S.~S.~Madani, Algebraic properties of ideals of poset homomorphisms, \href{http://arxiv.org/pdf/1505.07581v2.pdf}{arXiv:1505.07581}

\bibitem{OHH} H.~Ohsugi, J.~Herzog and T.~Hibi, Combinatorial pure subrings, Osaka J.~Math. {\bf 37} (2000), 745--757.

%
%
%
\bibitem{OhHirootsystem}
H.~Ohsugi and T.~Hibi,
Quadratic initial ideals of root systems,
{\em Proc. Amer. Math. Soc.} {\bf 130} (2002), 1913--1922.
%
%
%


\bibitem{N} W.~Bruns, B.~Ichim, NORMALIZ. Computing normalizations of affine semigroups. With contributions by C. S\"ager. Available at \href{http://www.math.uos.de/normaliz}{http://www.math.uos.de/normaliz}

\bibitem{St} B.~Sturmfels, Gr\"obner Bases and Convex Polytopes, Amer. Math. Soc., Providence, RI, 1995.


\bibitem{Vi} R.~H.~Villarreal, Monomial Algebras, Pure and Applied Mathematics, Marcel Dekker, 2001.
\end{thebibliography}
\end{document}